\documentclass[fleqn,reqno]{amsart}
\usepackage{url}
\usepackage{amssymb,amsmath,amssymb,color,cite}
\usepackage{amsmath,amsthm,amssymb, amstext,}
\usepackage{graphicx}
\usepackage{multirow}
\newtheorem{theorem}{Theorem}[section]

\newtheorem{corollary}{Corollary}[section]
\newtheorem{Conjecture}{Conjecture}[section]
\newtheorem{open problem}{\sc \bf Open problem}[section]
\newtheorem{remark}{Remark}[section]

\DeclareMathOperator{\sech}{sech}

\numberwithin{equation}{section} 
\begin{document}

\title[Some results  associated with Bernoulli and Euler numbers with applications]{Some results  associated with Bernoulli and Euler numbers with applications}

\author[{ C.-P. Chen }]{ Chao-Ping Chen$^{*}$}

\address{C.-P. Chen: School of Mathematics and Informatics, Henan Polytechnic University, Jiaozuo City 454000, Henan Province, China}
 \email{chenchaoping@sohu.com}

\author[{ R.B. Paris }]{Richard B. Paris}

\address{R.B. Paris: Division of Computing and Mathematics\\
 University of Abertay, Dundee, DD1 1HG, UK}
 \email{R.Paris@abertay.ac.uk}

\thanks{*Corresponding Author}



\begin{abstract}
In this paper, we  present  series
representations of the remainders in the expansions for $2/(e^t+1)$, $\sech t$ and $\coth t$.
 For example,
we prove that for  $t > 0$ and $N\in\mathbb{N}:=\{1, 2, \ldots\}$,
\begin{align*}\label{Thm1-remainder-cosh-OpenProblemsolution1}
\sech t=\sum_{j=0}^{N-1}\frac{E_{2j}}{(2j)!}t^{2j}+R_N(t)
\end{align*}
with
\begin{align*}
R_N(t)=\frac{(-1)^{N}2t^{2N}}{\pi^{2N-1}}\sum_{k=0}^{\infty}\frac{(-1)^{k}}{(k+\frac{1}{2})^{2N-1}\Big(t^2+\pi^2(k+\frac{1}{2})^2\Big)},
\end{align*}
and
\begin{align*}
\sech t=\sum_{j=0}^{N-1}\frac{E_{2j}}{(2j)!}t^{2j}+\Theta(t,
N)\frac{E_{2N}}{(2N)!}t^{2N}
\end{align*}
with a suitable $0 < \Theta(t, N) < 1$. Here $E_n$ are the Euler numbers. By using the obtained
results, we deduce some inequalities and completely monotonic
functions associated with the ratio of gamma functions. Furthermore, we give a (presumably new) quadratic recurrence relation for the Bernoulli numbers.
\end{abstract}


\maketitle

\section{Introduction}
 The Bernoulli polynomials $B_n(x)$ and Euler polynomials
$E_n(x)$ are defined, respectively, by the generating functions:
\begin{equation*}\label{generalizedBernoullipolynomials}
\frac{te^{xt}}{e^{t}-1}=\sum_{n=0}^{\infty}B_{n}(x)\frac{t^n}{n!}\quad
(|t|<2\pi)\quad \text{and}\quad
\frac{2e^{xt}}{e^{t}+1}=\sum_{n=0}^{\infty}E_{n}(x)\frac{t^n}{n!}\quad
(|t|<\pi).
\end{equation*}
The  numbers $B_n = B_n(0)$ and $E_n = 2^{n}E_n(\frac{1}{2})$, which are
known to be rational numbers   and integers, respectively,  are
called Bernoulli and Euler numbers.

It follows from \cite[Chapter 4, Part I, Problem 154]{PS}
 that
\begin{equation}\label{Problem}
\sum_{j=1}^{2m}\frac{B_{2j}}{(2j)!}t^{2j}<\frac{t}{e^{t}-1}-1+\frac{t}{2}
<\sum_{j=1}^{2m+1}\frac{B_{2j}}{(2j)!}t^{2j}
\end{equation}
for $t>0$ and $m\in \mathbb{N}_0:=\mathbb{N}\cup \{0\}, \,
\mathbb{N}:=\{1,2,3,\ldots\}$.  The inequality \eqref{Problem} can
be also found in \cite{Koumandos1458,Sasvari}.
It is also known   \cite[p. 64]{Temme1996} that
\begin{equation}\label{remainder3}
\frac{t}{e^{t}-1}-1
+\frac{t}{2}=\sum_{j=1}^{n}\frac{B_{2j}}{(2j)!}t^{2j}+(-1)^{n}t^{2n+2}\nu_{n}(t)\quad
(n\in \mathbb{N}_0),
\end{equation}
 where
\begin{equation}\label{remainder3-series}
\nu_{n}(t)=\frac{2}{(2\pi)^{2n}}\sum_{k=1}^{\infty}\frac{1}{k^{2n}(t^{2}+4\pi^{2}k^{2})}.
\end{equation}
It is easily seen that \eqref{remainder3} implies \eqref{Problem}. Koumandos \cite{Koumandos1458}  gave the following  integral
representation of $\nu_{n}(t)$:
\begin{equation}\label{remainder-integral}
\nu_{n}(t)=\frac{(-1)^{n}}{(2n+ 1)!}\frac{
1}{e^t-1}\int_{0}^{1}e^{xt}B_{2n+1}(x)\textup{\,d}x.
\end{equation}

\begin{remark}
From  \eqref{remainder-integral}, it is possible to deduce
\eqref{remainder3-series} by making use of the expansion \textup{\cite[p. 592, Eq. (24.8.2)]{Olver-Lozier-Boisvert-Clarks2010}}
\begin{equation*}\label{E2m-1-known}\begin{split}
B_{2n+1}(x)=
\frac{(-1)^{n+1}2(2n+1)!}{(2\pi)^{2n+1}}\sum_{k=1}^{\infty}\frac{\sin(2k\pi
x)}{k^{2n+1}}\quad (n\in\mathbb{N}, \quad 0\leq x\leq1).
\end{split}\end{equation*}
We then obtain from \eqref{remainder-integral} that
\begin{align*}
&\nu_{n}(t)=-\frac{
1}{e^t-1}\frac{2}{(2\pi)^{2n+1}}\sum_{k=1}^{\infty}\int_{0}^{1}\frac{e^{xt}\sin(2k\pi x)}{k^{2n+1}}\textup{\,d}x=\frac{2}{(2\pi)^{2n}}\sum_{k=1}^{\infty}\frac{1}{k^{2n}(t^{2}+4\pi^{2}k^{2})}.
\end{align*}
An alternative derivation of \eqref{remainder3} and another integral representation of the remainder
function $\nu_{n}(t)$ are given in the appendix.
\end{remark}

Binet's first  formula  \cite[p. 16]{Srivastava2001} for the logarithm of  $\Gamma(x)$ states that
\begin{equation}\label{binet}
\ln \Gamma(x)= \left(x-\frac1{2}\right)\ln x-x+\ln\sqrt{2\pi}+\int_{0}^{\infty}\left(\frac{t}{e^{t}-1}-1+\frac{t}{2}\right)\frac{e^{-xt}}{t^2}\textup{\,d} t \quad (x>0).
\end{equation}
Combining \eqref{remainder3} with \eqref{binet}, Xu and Han
\cite{Xu-Han47--51}  deduced in 2009 that for every
$m\in\mathbb{N}_0$, the function
\begin{equation}\label{Rm-CLF}
R_{m}(x)=(-1)^{m}\left[\ln \Gamma(x)-\left(x-\frac{1}{2}\right)\ln
x+x-\ln
\sqrt{2\pi}-\sum_{j=1}^{m}\frac{B_{2j}}{2j(2j-1)x^{2j-1}}\right]
\end{equation}
is completely monotonic on $(0, \infty)$. Recall that a function $f(x)$
is said to be completely monotonic on an interval $I$ if it has
derivatives of all orders on $I$ and satisfies the following
inequality:
\begin{equation}\label{cmf-dfn-ineq}
(-1)^{n}f^{(n)}(x)\geq0\quad  (x\in I, \quad n\in \mathbb{N}_0).
\end{equation}
 For $m=0$, the complete
monotonicity property of $R_{m}(x)$ was proved by Muldoon
\cite{Muldoon54}.  Alzer \cite{Alzer373} first proved in 1997 that
$R_{m}(x)$ is completely monotonic on $(0, \infty)$. In 2006,
 Koumandos \cite{Koumandos1458} proved the double inequality \eqref{Problem},
 and then used  \eqref{Problem} and \eqref{binet} to give a simpler  proof of the complete monotonicity property of $R_{m}(x)$. In 2009,
Koumandos and Pedersen \cite[Theorem 2.1]{Koumandos-Pedersen33--40}
strengthened this result.

Chen and  Paris \cite[Lemma 1]{Chen-Paris514--529}
presented an analogous result to \eqref{Problem} given by
\begin{equation}\label{ThmEn-inequality}\begin{split}
\sum_{j=2}^{2m+1}\frac{(1-2^{2j})B_{2j}}{j}\frac{t^{2j-1}}{(2j-1)!}<\frac{2}{e^t+1}-1+\frac{t}{2}&<\sum_{j=2}^{2m}\frac{(1-2^{2j})B_{2j}}{j}\frac{t^{2j-1}}{(2j-1)!}
\end{split}\end{equation}
for $t>0$ and $m\in\mathbb{N}$. The inequality  \eqref{ThmEn-inequality} can  also be written for $t>0$ and $m\in\mathbb{N}_0$ as
\begin{equation}\label{ThmEn-inequalityre}\begin{split}
(-1)^{m+1}\left(\frac{2}{e^t+1}-1-\sum_{j=1}^{m}\frac{(1-2^{2j})B_{2j}}{j}\frac{t^{2j-1}}{(2j-1)!}\right)>0.
\end{split}\end{equation}
 Based on the  inequality \eqref{ThmEn-inequalityre}, Chen and  Paris \cite[Theorem 1]{Chen-Paris514--529}
 proved
that for every $m\in\mathbb{N}_0$, the function
\begin{equation}\label{En-Fmx}\begin{split}
F_{m}(x)=(-1)^{m}\left[\ln\left(\frac{\Gamma(x+1)}{\Gamma(x+\frac{1}{2})}\right)-\frac{1}{2}\ln
x-\sum_{j=1}^{m}\left(1-\frac{1}{2^{2j}}\right)\frac{B_{2j}}{j(2j-1)x^{2j-1}}\right]
\end{split}\end{equation}
is completely monotonic on $(0,\infty)$. This result is similar to
the complete monotonicity property of $R_{m}(x)$ in \eqref{Rm-CLF}. In analogy with  \eqref{remainder3}, these authors also considered \cite[Eq. (2.4)]{Chen-Paris514--529} the remainder $r_{m}(t)$ in the expansion
\begin{equation}\label{remainder-formula-r_m(x)}
\frac{2}{e^t+1}=1+\sum_{j=1}^{m}\frac{(1-2^{2j})B_{2j}}{j\cdot(2j-1)!}t^{2j-1}+r_{m}(t)
\end{equation}
and gave an integral representation for $r_{m}(t)$ when $t>0$.

Chen \cite{Chen790--799} proposed the following conjecture.
\begin{Conjecture}\label{Conjecture1-Gamma-prod}
For $t>0$ and $m\in\mathbb{N}_0$, let
\begin{align}\label{conjecture-Gamma-prod}
\mu_m(t)&=\frac{e^{t/3}-e^{2t/3}}{e^{t}-1}-\sum_{j=0}^{m}\frac{2B_{2j+1}(\frac{1}{3})}{(2j+1)!}t^{2j}
\end{align}
and
\begin{align}\label{conjecture-Gamma-prod-nu}
\nu_m(t)&=\frac{e^{t/4}-e^{3t/4}}{e^{t}-1}-\sum_{j=0}^{m}\frac{2B_{2j+1}(\frac{1}{4})}{(2j+1)!}t^{2j},
\end{align}
 where $B_n(x)$ denotes the Bernoulli polynomials.
Then,  for $t>0$  and $m\in\mathbb{N}_0$,
\begin{equation}\label{Conjecture1-Gamma-mu}
(-1)^{m}\mu_m(t)>0
\end{equation}
and
\begin{align}\label{Conjecture1-Gamma-nu}
(-1)^{m}\nu_m(t)>0.
\end{align}
\end{Conjecture}
  Chen \cite[Lemma 1]{Chen790--799} has proved the statements in Conjecture
\ref{Conjecture1-Gamma-prod}  for $m=0, 1, 2$, and $3$.
He has also pointed out in \cite{Chen790--799} that, if Conjecture \ref{Conjecture1-Gamma-prod} is true,
then it follows that the functions
\begin{align}\label{Thm1-asymptotic-ratio-gammas-rewrittenfind}
U_m(x)=(-1)^{m}\left[\ln\frac{\Gamma(x+\frac{2}{3})}{x^{1/3}\Gamma(x+\frac{1}{3})}-
\sum_{j=1}^{m}\frac{B_{2j+1}(\frac{1}{3})}{j(2j+1)}\frac{1}{x^{2j}}\right]
\end{align}
and
\begin{align}\label{Thm1-asymptotic-ratio-gammas-rewrittenfind-Vx}
V_m(x)=(-1)^{m}\left[\ln\frac{\Gamma(x+\frac{3}{4})}{x^{1/2}\Gamma(x+\frac{1}{4})}-
\sum_{j=1}^{m}\frac{B_{2j+1}(\frac{1}{4})}{j(2j+1)}\frac{1}{x^{2j}}\right]
\end{align}
for $m\in\mathbb{N}_0$ are completely monotonic on $(0,\infty)$.
The complete monotonicity properties of $U_m(x)$ and $V_m(x)$ are
similar to the complete monotonicity property of $F_{m}(x)$ in
\eqref{En-Fmx}.

In this paper, we obtain the following results: (i) a series representation of the remainder $r_m(t)$ in \eqref{remainder-formula-r_m(x)} (Theorem \ref{Thm1-remainder}); (ii) a series representation of the remainder in the
expansion of $\sech t$ involving the Euler numbers (Theorem \ref{Thm2-remainder-cosh-OpenProblem}), together with the double
inequality for $t>0$ and $m\in\mathbb{N}_0$,
\begin{equation}\label{Euler-constantSm-inequality1}
\sum_{j=0}^{2m+1}\frac{E_{2j}}{(2j)!}t^{2j}<\sech t<\sum_{j=0}^{2m}\frac{E_{2j}}{(2j)!}t^{2j};
\end{equation}
(iii) the proof of the inequality \eqref{Conjecture1-Gamma-nu} for all $m\in\mathbb{N}_0$, and a demonstration that the
function $V_m(x)$ in \eqref{Thm1-asymptotic-ratio-gammas-rewrittenfind-Vx} is completely monotonic on $(0, \infty)$ (Remark \ref{Remark-completely-function-Vm}); (iv) a series
representation of the remainder in the expansion for $\coth t$ (Theorem \ref{Thm3-remainder-coth});  and finally, (v) a quadratic recurrence relation for the Bernoulli numbers (Theorem \ref{Thm4-remainder}).

\vskip 8mm

\section{Main results}
\begin{theorem}\label{Thm1-remainder}
For $t>0$ and $m\in\mathbb{N}$,
\begin{align}\label{Chen-remainder-rm}
\frac{2}{e^t+1}=1+\sum_{j=1}^{m}\frac{(1-2^{2j})B_{2j}}{j\cdot(2j-1)!}t^{2j-1}+(-1)^{m+1}t^{2m+1}s_m(t),
\end{align}
where  $s_m(t)$ is given by
\begin{align}\label{Chen-rm(x)}
s_m(t)=\frac{4}{\pi^{2m}}\sum_{k=0}^{\infty}\frac{1}{(2k+1)^{2m}\big(t^2+\pi^2(2k+1)^2\big)}.
\end{align}
\end{theorem}

\begin{proof}
 Boole's summation formula (see \cite[p. 17, Theorem 1.4]{Temme1996}) for a function $f(t)$ defined on $[0, 1]$
with $k$ continuous derivatives states that, for $k\in\mathbb{N}$,
\begin{equation}\label{Boole-summation-formula}\begin{split}
f(1)=\frac{1}{2}\sum_{j=0}^{k-1}\frac{E_{j}(1)}{j!}\Big(f^{(j)}(1)+f^{(j)}(0)\Big)+\frac{1}{2(k-1)!}\int_{0}^{1}f^{(k)}(x)E_{k-1}(x){\rm
d}x.
\end{split}\end{equation}
 Noting \cite[p. 590]{Olver-Lozier-Boisvert-Clarks2010} that
\begin{equation}\label{noting-En}
E_{n}(1)=\frac{2(2^{n+1}-1)}{n+1}B_{n+1}\qquad (n\in\mathbb{N}),
\end{equation}
 we see that
\begin{equation*}\label{noting-En}\begin{split}
E_{2j-1}(1)=\frac{(2^{2j}-1)B_{2j}}{j}\quad\text{and}\quad
E_{2j}(1)=0\quad  (j\in\mathbb{N}).
\end{split}\end{equation*}

The choice\footnote{It is also possible to choose $k=2m$ in (\ref{Boole-summation-formula}) and to use the Fourier expansion for $E_{2m+1}(x)$ in  \cite[p.~16]{Temme1996} to obtain the same result.}
 $k=2m+1$ in \eqref{Boole-summation-formula} yields
\begin{align}\label{Boole-summation-formula-k=2m+1}
f(1)-f(0)&=\sum_{j=1}^{m}\frac{(2^{2j}-1)B_{2j}}{j\cdot(2j-1)!}\Big(f^{(2j-1)}(1)+f^{(2j-1)}(0)\Big)\nonumber\\
&\qquad\qquad\quad+\frac{1}{(2m)!}\int_{0}^{1}f^{(2m+1)}(x)E_{2m}(x){\rm
d}x.
\end{align}

Application of the above formula to $f(x)=e^{xt}$
then produces
\begin{equation}\label{Boole-summation-formula-re-obtainnew}\begin{split}
\frac{2}{e^t+1}=1+\sum_{j=1}^{m}\frac{(1-2^{2j})B_{2j}}{j\cdot(2j-1)!}t^{2j-1}+r_m(t),
\end{split}\end{equation}
where
\begin{equation}\label{Boole-summation-formula-re-obtainNew}\begin{split}
r_m(t)=-\frac{1}{e^t+1}\frac{t^{2m+1}}{(2m)!}\int_{0}^{1}e^{xt}E_{2m}(x){\rm
d}x.
\end{split}\end{equation}
 Using the following formula \textup{(}see \textup{\cite[p.
 16]{Temme1996}):}
\begin{equation}\label{Euler-constant-Fourier-Expansion}
E_{2m}(x)= (-1)^{m}\frac{4(2m)!}{\pi^{2m+1}}\sum_{k=0}^{\infty}\frac{\sin[(2k+1)\pi
x]}{(2k+1)^{2m+1}}\qquad (m\in\mathbb{N}, \quad 0\leq x\leq1),
\end{equation}
we obtain
\begin{align*}
r_m(t)&=\frac{(-1)^{m+1}}{e^t+1} \frac{4t^{2m+1}}{\pi^{2m+1}}\sum_{k=0}^{\infty}\int_{0}^{1}e^{xt}\frac{\sin[(2k+1)\pi x]}{(2k+1)^{2m+1}}{\rm d}x\nonumber\\
&=(-1)^{m+1}\frac{4t^{2m+1}}{\pi^{2m+1}}\sum_{k=0}^{\infty}\frac{1}{(2k+1)^{2m}\big(t^2+\pi^2(2k+1)^2\big)}.
\end{align*}
This completes the proof of Theorem \ref{Thm1-remainder}.
\end{proof}

\begin{remark}
From \eqref{Chen-remainder-rm} we retrieve
\eqref{ThmEn-inequalityre}.
\end{remark}

\begin{remark}
From  \textup{\cite[p. 592, Eq. (24.7.9)]{Olver-Lozier-Boisvert-Clarks2010}} and \textup{\cite[p. 43, Ex. 12(i)]{Wang-Guo1999}} we have
\begin{align*}
E_{2n}(x)=(-1)^n \sin(\pi x) \int_0^\infty \frac{4t^{2n} \cosh(\pi t)}{\cosh(2\pi t)-\cos(2\pi x)}\, {\rm d}t
\qquad (0<x<1,\quad n\in \mathbb{N}_0),
\end{align*}
from which it follows that
\begin{align*}
E_{4m}(x)>0\quad \mbox{and}\quad E_{4m+2}(x)<0\qquad (
0<x<1,\quad m\in \mathbb{N}_0).
\end{align*}
By combining these inequalities with \eqref{Boole-summation-formula-re-obtainnew}
and \eqref{Boole-summation-formula-re-obtainNew} we immediately
obtain \eqref{ThmEn-inequality}.
\end{remark}

\begin{corollary}\label{Thm2-remainder}
For $t>0$ and $m\in\mathbb{N}$,
\begin{align}\label{diff-rm(x)obtain}
(-1)^{m}\left(\frac{2e^t}{(e^t+1)^2}-\sum_{j=1}^{m}\frac{(2^{2j}-1)B_{2j}}{j\cdot(2j-2)!}t^{2j-2}\right)>0.
\end{align}
\end{corollary}

\begin{proof}
Differentiating  the expression in \eqref{Chen-remainder-rm}, we find
\begin{align}\label{diff-rm(x)}
-\frac{2}{(e^t+1)^2}e^t=-\sum_{j=1}^{m}\frac{(2^{2j}-1)B_{2j}}{j\cdot(2j-2)!}t^{2j-2}+(-1)^{m+1}\big(t^{2m+1}s_m(t)\big)'.
\end{align}
It is easy to see that
\begin{align*}
t^2s_m(t)+s_{m-1}(t)=\frac{4}{\pi^{2m}}\sum_{k=0}^{\infty}\frac{1}{(2k+1)^{2m}}=\frac{4}{\pi^{2m}}(1-2^{-2m})\zeta(2m),
\end{align*}
 where $\zeta(z)$ is
the Riemann zeta function. This last expression can be written as
\begin{align}\label{Chen-rm(x)-rm-1(x)}
t^2s_m(t)=\frac{4}{\pi^{2m}}(1-2^{-2m})\zeta(2m)-s_{m-1}(t).
\end{align}
Then, since $s_m(t)$ is  strictly decreasing for $t>0$, we deduce from
\eqref{Chen-rm(x)-rm-1(x)} that $t^2s_m(t)$ is strictly increasing
for $t>0$. Hence,  $t^{2m+1}s_m(t)$ is strictly increasing for
$t>0$, and we then obtain from \eqref{diff-rm(x)} that
\begin{align*}
(-1)^{m}\left(\frac{2e^t}{(e^t+1)^2}-\sum_{j=1}^{m}\frac{(2^{2j}-1)B_{2j}}{j\cdot(2j-2)!}t^{2j-2}\right)=\big(t^{2m+1}s_m(t)\big)'>0
\end{align*}
for $t>0$ and $m\in\mathbb{N}$. The proof is complete.
\end{proof}


\begin{theorem}\label{Thm2-remainder-cosh-OpenProblem}
For  $t > 0$ and $N\in\mathbb{N}$, we have
\begin{align}\label{Thm1-remainder-cosh-OpenProblemsolution1}
\sech t=\sum_{j=0}^{N-1}\frac{E_{2j}}{(2j)!}t^{2j}+R_N(t)
\end{align}
with
\begin{align}
R_N(t)=\frac{(-1)^{N}2t^{2N}}{\pi^{2N-1}}\sum_{k=0}^{\infty}\frac{(-1)^{k}}{(k+\frac{1}{2})^{2N-1}\Big(t^2+\pi^2(k+\frac{1}{2})^2\Big)},
\end{align}
and
\begin{align}\label{Solution-open}
\sech t=\sum_{j=0}^{N-1}\frac{E_{2j}}{(2j)!}t^{2j}+\Theta(t,
N)\frac{E_{2N}}{(2N)!}t^{2N}
\end{align}
with a suitable $0 < \Theta(t, N) < 1$.
\end{theorem}

\begin{proof}
It follows from \cite[p. 136]{Whittaker-Watson1966} (see also
\cite[p. 458, Eq. (27.3)]{Berndt1998}) that
\begin{equation*}\label{Euler-constantSm-inequality1}
\frac{\pi}{4\cosh\left(\frac{\pi
x}{2}\right)}=\sum_{k=0}^{\infty}\frac{(-1)^{k}(2k+1)}{(2k+1)^{2}+x^2},
\end{equation*}
which can be written as
\begin{equation}\label{cosh-expansion}
\sech t=\frac{4}{\pi}\sum_{k=0}^{\infty}\frac{(-1)^{k}}{(2k+1)\left(1+\left(\frac{2t}{\pi(2k+1)}\right)^2\right)}.
\end{equation}

Substitution of $x=\frac{1}{2}$ in \eqref{Euler-constant-Fourier-Expansion} leads to
\begin{equation}\label{Euler-constant-Fourier-Expansion-obtain}
\sum_{k=0}^{\infty}\frac{(-1)^k}{(2k+1)^{2j+1}}=\frac{(-1)^{j}\pi^{2j+1}}{2^{2j+2}(2j)!}\,E_{2j}.
\end{equation}
Using the identity
\begin{equation}\label{identity1/(1+z)}
\frac{1}{1+q}=\sum_{j=0}^{N-1}(-1)^{j}q^{j}+(-1)^{N}\frac{q^N}{1+q}\qquad
(q\not=-1)
\end{equation}
and \eqref{Euler-constant-Fourier-Expansion-obtain}, we obtain from
\eqref{cosh-expansion} that
\begin{align*}
\sech t&=\frac{4}{\pi}\sum_{k=0}^{\infty}\frac{(-1)^{k}}{(2k+1)}\left(\sum_{j=0}^{N-1}(-1)^{j}\left(\frac{2t}{\pi(2k+1)}\right)^{2j}+(-1)^{N}\frac{\left(\frac{2t}{\pi(2k+1)}\right)^{2N}}{1+\left(\frac{2t}{\pi(2k+1)}\right)^2}\right)\nonumber\\
&=\sum_{j=0}^{N-1}\frac{E_{2j}}{(2j)!}t^{2j}+R_N(t),
\end{align*}
with
\begin{align*}
R_N(t)=\frac{2}{\pi^{2N-1}}\sum_{k=0}^{\infty}\frac{(-1)^{N+k}}{(k+\frac{1}{2})^{2N-1}}\,\frac{t^{2N}}{\big(t^2+\pi^2(k+\frac{1}{2})^2\big)}.
\end{align*}

Noting that \eqref{Euler-constant-Fourier-Expansion-obtain} holds,
we find that $R_N(t)$ can be written as
\begin{align*}
R_N(t)=\Theta(t,N)\,\frac{E_{2N} t^{2N}}{(2N)!},\qquad \Theta(t,N):=\frac{F(t)}{F(0)},
\end{align*}
where
\begin{align*}
F(t):=\sum_{k=0}^\infty (-1)^k \alpha_k,\qquad \alpha_k:=\frac{1}{(k+\frac{1}{2})^{2N-1}}\,\frac{1}{t^2+\pi^2(k+\frac{1}{2})^2}.
\end{align*}
Then it is easily seen that $\alpha_{2k}>\alpha_{2k+1}$ for $k\in\mathbb{N}_0$, $t>0$ and $N\in\mathbb{N}$; thus $F(t)>0$ for $t>0$.
Differentiation yields
\begin{align*}
F'(t)=-2t\sum_{k=0}^{\infty}\frac{(-1)^{k}\alpha_k}{t^2+\pi^2(k+\frac{1}{2})^2}
\end{align*}
and a similar reasoning shows that $F'(t)<0$ for $t>0$.
Hence, for all $t > 0$ and $N\in\mathbb{N}$, we have $0 < F (t) < F(0)$ and thus $0 < \Theta(t,
N) < 1$. The proof of Theorem \ref{Thm2-remainder-cosh-OpenProblem} is complete.
\end{proof}

\begin{remark}
Recalling that
\begin{align*}
E_{4m} >0\quad \mbox{and}\quad  E_{4m+2}<0\qquad (  m\in
\mathbb{N}_0),
\end{align*}
we can deduce \eqref{Euler-constantSm-inequality1}
from \eqref{Solution-open}. Note that the inequality
\eqref{Euler-constantSm-inequality1} can also be written as
\begin{equation}\label{Euler-constantSm-inequality-ie}
(-1)^{m+1}\left(\sech
t-\sum_{j=0}^{m}\frac{E_{2j}}{(2j)!}t^{2j}\right)>0\qquad (t>0,
\,\, m\in\mathbb{N}_0).
\end{equation}
\end{remark}

\begin{remark}\label{Remark-completely-function-Vm}
It was shown  in \cite{Chen790--799} that
\eqref{conjecture-Gamma-prod-nu} can be written as
\begin{equation}
\nu_m(t)=-\frac{1}{2\cosh(\frac{t}{4})}+\sum_{j=0}^{m}\frac{E_{2j}}{2(2j)!}\left(\frac{t}{4}\right)^{2j}
\end{equation}
and  \eqref{Conjecture1-Gamma-nu} is equivalent to
\eqref{Euler-constantSm-inequality-ie}. Hence,  for $t>0$  and
$m\in\mathbb{N}_0$, \eqref{Conjecture1-Gamma-nu} holds true.

It was also shown in \cite{Chen790--799} that
\begin{align}\label{Vm-gammaRatio1/4}
V_{m}(x)&=(-1)^m\Bigg[\int_{0}^{\infty}\left(\frac{e^{t/4}-e^{3t/4}}{e^{t}-1}+\frac{1}{2}\right)\frac{e^{-xt}}{t}{\rm
d}t-\sum_{j=1}^{m}\frac{2B_{2j+1}(\frac{1}{4})}{(2j+1)!}\int_0^\infty
t^{2j-1}e^{-xt}{\rm d}
t\Bigg]\notag\\
&=\int_{0}^{\infty}(-1)^{m}\nu_m(t)\frac{e^{-xt}}{t}\textup{\,d}t.
\end{align}
We obtain from  \eqref{Vm-gammaRatio1/4}
 that for all $m\in\mathbb{N}_0$,
\begin{align*}
(-1)^{n}V_m^{(n)}(x)=\int_{0}^{\infty}(-1)^{m}\nu_m(t)t^{n-1}e^{-xt}\textup{\,d}t>0
\end{align*}
 for $x>0$ and $n\in\mathbb{N}_0$. Hence, the function
$V_m(x)$, defined by
\eqref{Thm1-asymptotic-ratio-gammas-rewrittenfind-Vx},   is
completely monotonic on $(0,\infty)$.
\end{remark}

Sondow and Hadjicostas \cite{Sondowa292--314} introduced and studied
the generalized-Euler-constant function $\gamma(z)$, defined by
 \begin{align}\label{gammaz}
\gamma(z)=\sum_{n=1}^{\infty}z^{n-1}\left(\frac{1}{n}-\ln
\frac{n+1}{n}\right),
 \end{align}
where the series converges when $|z|\leq 1$. Pilehrood and Pilehrood
 \cite{Pilehrood117--131} considered the function $z\gamma(z)$ ($|z|\leq1$). The function
$\gamma(z)$ generalizes both Euler's constant $\gamma(1)$ and the
alternating Euler constant $\ln \frac{4}{\pi} = \gamma(-1)$
\cite{Sondow61--65, Sondow2005}. An interesting comparison by  Sondow \cite{Sondow61--65} is the double integral and alternating series
\begin{equation}\label{Sondow-double-integral}\begin{split}
\ln\frac{4}{\pi}=\int_{0}^{1}\int_{0}^{1}\frac{x-1}{(1+xy)\ln(xy)}\textup{d}x\textup{d}y=\sum_{n=1}^{\infty}(-1)^{n-1}\left(\frac{1}{n}-\ln
\frac{n+1}{n}\right).
\end{split}\end{equation}

The formula \eqref{Vm-gammaRatio1/4} can provide integral
representations for the constant $\pi$. For example,
the choice $(x,m)=(1/4, 0)$ in \eqref{Vm-gammaRatio1/4}  yields
\begin{equation}\label{new-representationsPi2new}\begin{split}
\int_{0}^{\infty}\left(\frac{e^{t/4}-e^{3t/4}}{e^{t}-1}+\frac{1}{2}\right)\frac{2e^{-t/4}}{t}{\rm
d}t=\ln\frac{4}{\pi},
\end{split}\end{equation}
which provides a new  integral  representation for the
alternating Euler constant $\ln \frac{4}{\pi}$.
The choice $(x,m)=(3/4, 0)$ in \eqref{Vm-gammaRatio1/4}  yields
\begin{equation}\label{new-representationsPi2newnew}\begin{split}
\int_{0}^{\infty}\left(\frac{e^{t/4}-e^{3t/4}}{e^{t}-1}+\frac{1}{2}\right)\frac{2e^{-3t/4}}{t}{\rm
d}t=\ln\frac{\pi}{3}.
\end{split}\end{equation}
Many formulas  exist for the representation of $\pi$, and a
collection of these formulas is listed in
\cite{Sofo184--189,SofoJIPAM2005}. For more history of $\pi$ see
\cite{Beckmann1971, Berggren-Borwein-Borwein, Dunham1990}.

Noting \cite[Eq. (3.26)]{Chen790--799} that $B_{2n+1}(\tfrac{1}{4})$
can be expressed in terms of the Euler numbers
\begin{align}\label{B1/4-E2n}
B_{2n+1}(\tfrac{1}{4})=-\frac{(2n+1)E_{2n}}{4^{2n+1}}\qquad (n\in\mathbb{N}_0),
\end{align}
we find that \eqref{Thm1-asymptotic-ratio-gammas-rewrittenfind-Vx}
can be written as
\begin{align}\label{Thm1-asymptotic-ratio-gammas-rewrittenfindVxEuler-numbers}
V_m(x)=(-1)^{m}\left[\ln\frac{\Gamma(x+\frac{3}{4})}{x^{1/2}\Gamma(x+\frac{1}{4})}+
\sum_{j=1}^{m}\frac{E_{2j}}{j\cdot4^{2j+1}}\frac{1}{x^{2j}}\right].
\end{align}
From the inequalities $V_m(x)>0$  for $x>0$, we obtain the
following
\begin{corollary}\label{corollary1-Gamma-1/3-2/3-ratio-inequalityVx}
For $x>0$,
\begin{align}\label{Thm1-asymptotic-ratio-gammas-rewrittenfindVx}
x^{1/2}\exp\left(-\sum_{j=1}^{2m}\frac{E_{2j}}{j\cdot4^{2j+1}}\frac{1}{x^{2j}}\right)<\frac{\Gamma(x+\frac{3}{4})}{\Gamma(x+\frac{1}{4})}<
x^{1/2}\exp\left(-\sum_{j=1}^{2m+1}\frac{E_{2j}}{j\cdot4^{2j+1}}\frac{1}{x^{2j}}\right).
\end{align}
\end{corollary}

The problem of finding new and sharp inequalities for the gamma
function $\Gamma$ and, in particular, for the Wallis ratio
\begin{equation}\label{WallisRatio}\begin{split}
\frac{(2n-1)!!}{(2n)!!}=\frac{1}{\sqrt{\pi}}\frac{\Gamma(n+\frac{1}{2})}{\Gamma(n+1)}
 \end{split}\end{equation}
 has attracted the attention of many researchers
(see
\cite{Chen-Paris514--529,ChenQi397-401,Koumandos1365--1367,Lampret328--339,Lampret775--787,
Mortici425--433} and references therein).
 Here, we employ the special double
     factorial notation as follows:
    \begin{align*}
        &(2n)!!=2\cdot 4\cdot 6\cdots (2n)=2^n n!,\quad 0!!=1,\qquad (-1)!!=1,\\
        &(2n-1)!!=1\cdot 3\cdot 5\cdots
            (2n-1)=\pi^{-1/2}2^n\Gamma\left(n+\frac{1}{2}\right);
    \end{align*}
see \cite[p. 258]{abram}. For example, Chen and Qi \cite{ChenQi397-401} proved that for  $n\in\mathbb{N}$,
\begin{equation}
\label{walthmin}
\frac1{\sqrt{\pi\bigl(n+\frac{4}{\pi}-1\bigr)}}\le\frac{(2n-1)!!}{(2n)!!}<\frac1{\sqrt{\pi\bigl(n+\frac14\bigr)}},
\end{equation}
where the constants $\frac{4}{\pi}-1$ and $\frac14$ are the best
possible. This inequality is a consequence
of the complete monotonicity on  $(0, \infty)$ of the function (see \cite{Chen-Qi303--307})
\begin{equation}\label{WallisRatioVx}
V(x)=\frac{\Gamma(x+1)}{\sqrt{x+\frac{1}{4}}\,\Gamma(x+\frac{1}{2})}.
\end{equation}

If we write  \eqref{Thm1-asymptotic-ratio-gammas-rewrittenfindVx} as
\begin{align*}
\frac{1}{\sqrt{x}}\exp\left(\sum_{j=1}^{2m+1}\frac{E_{2j}}{j\cdot4^{2j+1}}\frac{1}{x^{2j}}\right)<\frac{\Gamma(x+\frac{1}{4})}{\Gamma(x+\frac{3}{4})}<
\frac{1}{\sqrt{x}}\exp\left(\sum_{j=1}^{2m}\frac{E_{2j}}{j\cdot4^{2j+1}}\frac{1}{x^{2j}}\right)
\end{align*}
and replace  $x$ by $x+\frac{1}{4}$, we find
\begin{align}\label{Thm1-asymptotic-ratio-gammas-rewrittenfindVxwrite-yields}
&\frac{1}{\sqrt{x+\frac{1}{4}}}\exp\left(\sum_{j=1}^{2m+1}\frac{E_{2j}}{j\cdot4^{2j+1}}\frac{1}{(x+\frac{1}{4})^{2j}}\right)<\frac{\Gamma(x+\frac{1}{2})}{\Gamma(x+1)}\nonumber\\
&\qquad\qquad\qquad\qquad\qquad\qquad\qquad<\frac{1}{\sqrt{x+\frac{1}{4}}}\exp\left(\sum_{j=1}^{2m}\frac{E_{2j}}{j\cdot4^{2j+1}}\frac{1}{(x+\frac{1}{4})^{2j}}\right).
\end{align}
Noting that \eqref{WallisRatio} holds, we then deduce from
\eqref{Thm1-asymptotic-ratio-gammas-rewrittenfindVxwrite-yields}
that
\begin{align}\label{Thm1-asymptotic-ratio-gammas-rewrittenfindVxwrite-yieldsWallis-ineq}
&\frac{1}{\sqrt{\pi(x+\frac{1}{4})}}\exp\left(\sum_{j=1}^{2m+1}\frac{E_{2j}}{j\cdot4^{2j+1}}\frac{1}{(x+\frac{1}{4})^{2j}}\right)<\frac{(2n-1)!!}{(2n)!!}\nonumber\\
&\qquad\qquad\qquad\qquad\qquad\qquad\qquad<\frac{1}{\sqrt{\pi(x+\frac{1}{4})}}\exp\left(\sum_{j=1}^{2m}\frac{E_{2j}}{j\cdot4^{2j+1}}\frac{1}{(x+\frac{1}{4})^{2j}}\right),
\end{align}
which generalizes  a recently published result of Chen \cite[Eq.
(3.40)]{Chen790--799}, who proved the inequality
\eqref{Thm1-asymptotic-ratio-gammas-rewrittenfindVxwrite-yieldsWallis-ineq}
for $m=1$.

\begin{theorem}\label{Thm3-remainder-coth}
For  $t > 0$ and $N\in\mathbb{N}_0$, we have
\begin{align}\label{Thm2-remainder-coth-expansion}
\coth t=\sum_{j=0}^{N}\frac{2^{2j}B_{2j}}{(2j)!}t^{2j-1}+\sigma_N(t),
\end{align}
where
\begin{align}\label{Thm2-remainder-coth-rN}
\sigma_N(t)=\frac{(-1)^{N}t^{2N+1}}{\pi^{2N}}\sum_{k=1}^{\infty}\frac{2}{k^{2N}(t^2+\pi^2k^2)},
\end{align}
and
\begin{align}\label{Thm2-remainder-coth-theta}
\coth t=\sum_{j=0}^{N}\frac{2^{2j}B_{2j}}{(2j)!}t^{2j-1}+\theta(t,
N)\frac{2^{2N+2}B_{2N+2}}{(2N+2)!}t^{2N+1}
\end{align}
with a suitable $0<\theta(t, N)<1$.
\end{theorem}

\begin{proof}
It follows from \cite[p. 126, Eq.
(4.36.3)]{Olver-Lozier-Boisvert-Clarks2010} that
\begin{equation}\label{coth-expansion}
\coth
t=\frac{1}{t}+2t\sum_{k=1}^{\infty}\frac{1}{\pi^2k^2+t^2}=\frac{1}{t}+\frac{2t}{\pi^2}\sum_{k=1}^{\infty}\frac{1}{k^2\big(1+(\frac{t}{\pi
k})^2\big)}.
\end{equation}
It is well known   that
\begin{equation}\label{Zeta-2n}
\sum_{k=1}^{\infty}\frac{1}{k^{2j}}=\frac{(-1)^{j-1}(2\pi)^{2j}B_{2j}}{2(2j)!}.
\end{equation}
Using \eqref{identity1/(1+z)} and \eqref{Zeta-2n}, we obtain from
\eqref{coth-expansion} that
\begin{align*}
\coth
t&=\frac{1}{t}+2t\sum_{k=1}^{\infty}\frac{1}{k^2\pi^2}\left(\sum_{j=0}^{N-1}(-1)^{j}\left(\frac{t}{k\pi}\right)^{2j}+(-1)^{N}\frac{\left(\frac{t}{k\pi}\right)^{2N}}{1+\left(\frac{t}{k\pi}\right)^2}\right)\nonumber\\
&=\frac{1}{t}+\sum_{j=0}^{N-1}\frac{2^{2j+2}B_{2j+2}}{(2j+2)!}t^{2j+1}+\sigma_N(t)\\
&=\sum_{j=0}^{N}\frac{2^{2j}B_{2j}}{(2j)!}t^{2j-1}+\sigma_N(t)
\end{align*}
with
\begin{align*}
\sigma_N(t)=\frac{2(-1)^{N}}{\pi^{2N}}\sum_{k=1}^{\infty}\frac{t^{2N+1}}{k^{2N}(t^2+\pi^2k^2)}.
\end{align*}

Noting that \eqref{Zeta-2n} holds, we can rewrite $\sigma_N(t)$ as
\begin{align*}
\sigma_N(t)=\theta(t,N)\,\frac{2^{2N+2} B_{2N+2}}{(2N+2)!}\,t^{2N+1},
\end{align*}
where
\begin{align*}
\theta(t,N):=\frac{f(t)}{f(0)},\qquad f(t):=\sum_{k=1}^\infty \frac{1}{k^{2N}(t^2+\pi^2k^2)}.
\end{align*}
Obviously, $f(t)>0$ and is strictly decreasing for $t>0$. Hence, for all
$t > 0$, $0 < f (t) < f (0)$ and thus $0 < \theta(t, N) < 1$. The
proof of Theorem \ref{Thm3-remainder-coth} is complete.
\end{proof}

The following expansion for Barnes $G$-function
was established by Ferreira and L\'opez \cite[Theorem
1]{Ferreira298-314}. For $|\textup{arg}(z)|<\pi$,
\begin{equation*}\label{Ferreira-Lopez-Formula}\begin{split}
\ln G(z+1)&=\frac{1}{4}z^{2}+z\ln
\Gamma(z+1)-\left(\frac{1}{2}z^{2}+\frac{1}{2}z+\frac{1}{12}\right)\ln z-\ln A\\
&\quad+\sum_{k=1}^{N-1}\frac{B_{2k+2}}{2k(2k+1)(2k+2)z^{2k}}+\mathcal{R}_{N}(z)\qquad
(N\in\mathbb{N}),
\end{split}\end{equation*}
where $B_{2k+2}$ are the Bernoulli numbers and $A$ is the
Glaisher--Kinkelin constant defined by
\begin{equation}\label{An}
\ln A=\lim_{n \to \infty}\left\{\ln
\left(\prod_{k=1}^{n}k^{k}\right)-\left(\frac{n^{2}}{2}+\frac{n}{2}+\frac{1}{12}\right)\ln
n+\frac{n^{2}}{4}\right\},
\end{equation}
the numerical value of $A$ being $1.282427\ldots$. For $\Re(z)>0$, the remainder
$\mathcal{R}_{N}(z)$ is  given by
\begin{equation}\label{G-RN}
\mathcal{R}_{N}(z)=\int_{0}^{\infty}\left(\frac{t}{e^{t}-1}-\sum_{k=0}^{2N}\frac{B_{k}}{k!}t^{k}\right)\frac{e^{-zt}}{t^{3}}d
t.
\end{equation}
Estimates for $|\mathcal{R}_{N}(z)|$ were also obtained by Ferreira and L\'opez
\cite{Ferreira298-314}, showing that the expansion is indeed an
asymptotic expansion of $\ln G(z+1)$ in sectors of the complex plane
cut along the negative axis. Pedersen \cite[Theorem
1.1]{Pedersen171-178} proved that for any $N\geq 1$, the function $
x\mapsto (-1)^{N}\mathcal{R}_{N}(x)$ is completely monotonic on $(0, \infty)$.

Here, we present another proof of this complete monotonicity result.
From \eqref{Thm2-remainder-coth-expansion}, we obtain the following
inequality:
\begin{align*}
\sum_{j=0}^{2m}\frac{2^{2j}B_{2j}}{(2j)!}t^{2j-1}<\coth
t<\sum_{j=0}^{2m+1}\frac{2^{2j}B_{2j}}{(2j)!}t^{2j-1}\qquad (t>0,
\,\, m\in\mathbb{N}_0),
\end{align*}
which is equivalent to
\begin{align}\label{Thm2-remainder-coth-expansion-inequality-ie}
(-1)^{N}\left(\coth
t-\sum_{j=0}^{N}\frac{2^{2j}B_{2j}}{(2j)!}t^{2j-1}\right)>0\qquad
(t>0, \,\, N\in\mathbb{N}_0).
\end{align}
Replacement of $t$ by $t/2$ in
\eqref{Thm2-remainder-coth-expansion-inequality-ie} yields
\begin{align}\label{Thm2-remainder-coth-expansion-inequality-or}
(-1)^{N}\left(\frac{t}{2}\coth
\left(\frac{t}{2}\right)-\sum_{j=0}^{N}\frac{B_{2j}}{(2j)!}t^{2j}\right)>0\qquad
(t>0, \,\, N\in\mathbb{N}_0).
\end{align}
Accordingly, we obtain from \eqref{G-RN}  that the function
\begin{equation}\label{G-RN-x}
(-1)^{N}\mathcal{R}_{N}(x)=\int_{0}^{\infty}(-1)^{N}\left(\frac{t}{2}\coth\left(\frac{t}{2}\right)-\sum_{k=0}^{N}\frac{B_{2k}}{(2k)!}t^{2k}\right)\frac{e^{-xt}}{t^{3}}d
t
\end{equation}
is completely monotonic on $(0, \infty)$.

\begin{remark}\label{Remark-deducte-remainder3again}
From \eqref{Thm2-remainder-coth-expansion}, we can deduce
\eqref{remainder3}.  In fact, noting that
\begin{align*}
\coth t=\frac{e^t+e^{-t}}{e^t-e^{-t}}=1+\frac{2}{e^{2t}-1},
\end{align*}
we see that \eqref{Thm2-remainder-coth-expansion} can be written  as
\begin{align}\label{Thm2-remainder-coth-expansion-re}
x+\frac{2x}{e^{2x}-1}=\sum_{j=0}^{N}\frac{B_{2j}}{(2j)!}(2x)^{2j}+\frac{(-1)^{N}x^{2N+2}}{\pi^{2N}}\sum_{k=1}^{\infty}\frac{2}{k^{2N}(x^2+\pi^2k^2)}.
\end{align}
Replacement of $x$ by $t/2$ in
\eqref{Thm2-remainder-coth-expansion-re} yields \eqref{remainder3}.
\end{remark}

\vskip 8mm

\section{A quadratic recurrence relation for $B_n$}
Euler (see \cite[p. 595, Eq. (24.14.2)]{Olver-Lozier-Boisvert-Clarks2010} and \cite{Wikipedia-contributors-Bernoulli-number}) presented  a quadratic recurrence relation for the Bernoulli numbers:
\begin{align}\label{Euler-quadratic-recurrences-Bn}
\sum_{k=0}^{n}\binom{n}{k} B_{k}B_{n-k}=(1-n)B_n-nB_{n-1}\qquad (n\geq1),
\end{align}
which  is equivalent to
\begin{align}\label{Euler-quadratic-recurrences-Bn-equivalent}
\sum_{j=1}^{n-1}\binom{2n}{2j} B_{2j}B_{2n-2j}=-(2n+1)B_{2n} \qquad (n\geq2).
\end{align}
The relation \eqref{Euler-quadratic-recurrences-Bn-equivalent} can be used to show by induction that
\begin{align*}
(-1)^{n-1}B_{2n}>0 \quad \text{for all}\quad n\geq1,
\end{align*}
i.e., the even-index Bernoulli numbers have alternating signs. Other  quadratic recurrences for the Bernoulli numbers have been given
by Gosper (see \cite[Eq. (38)]{Weisstein-Bn}) as
\begin{align}\label{Gosper-quadratic-recurrences-Bn}
B_n=\frac{1}{1-n}\sum_{k=0}^n (1-2^{1-k})(1-2^{k-n+1})\binom{n}{k} B_{k}B_{n-k}
\end{align}
and by Matiyasevitch \cite{Matiyasevich1997}    (see also \cite{Wikipedia-contributors-Bernoulli-number}) as
\begin{align}\label{Matiyasevitch-quadratic-recurrences-Bn}
B_n=\frac{1}{n(n+1)}\sum_{k=2}^{n-2} \left\{n+2-2\binom{n+2}{k} \right\} B_{k}B_{n-k}\qquad (n\geq 4).
\end{align}

Here, we present a (presumably new)  quadratic recurrence relation for the Bernoulli numbers.
\begin{theorem}\label{Thm4-remainder}
The  Bernoulli numbers satisfy the following quadratic recurrence relation:
\begin{align}\label{Thm3-Result}
B_{n}=\frac{1}{2^{n}-1}\sum_{k=2}^{n-2}(1-2^k)\binom{n}{k}B_kB_{n-k}           \quad \quad (n\geq4).
\end{align}
\end{theorem}

\begin{proof}
If we replace $t$ by $t/2$ in \eqref{Chen-remainder-rm}, we find
\begin{equation}\label{Write-Ej-as-bj}
\frac{2}{e^{t/2}+1}=1+\sum_{j=2}^{\infty}b_{j}t^{j-1}, \qquad b_j=\frac{2(1-2^{j})B_{j}}{2^{j-1}\cdot j!}.
\end{equation}
The Bernoulli numbers $B_n$ are defined by the generating function
\begin{equation}\label{generalizedBernoullinumbers}
\frac{t}{e^{t}-1}=\sum_{n=0}^{\infty}B_{n}\frac{t^n}{n!},
\end{equation}
which yields
\begin{equation}\label{Write-Bj-as}
\frac{t/2}{e^{t/2}-1}=\sum_{k=0}^{\infty}\frac{B_k t^k}{2^k k!}.
\end{equation}

It then follows from (\ref{Write-Ej-as-bj}) and (\ref{Write-Bj-as}) that
\begin{align*}
\frac{t}{e^{t}-1}&=\left(1+\sum_{j=2}^{\infty}b_{j}t^{j-1}\right)\sum_{k=0}^{\infty}\frac{B_k t^k}{2^{k} k!}\\
&=\sum_{k=0}^{\infty}\frac{B_k t^k}{2^{k} k!}+\sum_{j=2}^{\infty}b_jt^{j-1}\sum_{k=0}^{\infty}\frac{B_k t^k}{2^{k} k!}\\
&=\sum_{j=0}^{\infty}\frac{B_j t^j}{2^{j}
j!}+\sum_{j=1}^{\infty}\sum_{k=0}^{j-1}b_{k+2}\frac{B_{j-k-1} t^k}{2^{j-k-1}(j-k-1)!},
\end{align*}
that is
\begin{equation}\label{Chen-remainder-rmwriteas-t-2tthatisbj}
\frac{t}{e^{t}-1}=\sum_{j=0}^{\infty}\left(\frac{B_j}{2^{j}
j!}+\sum_{k=0}^{j-1}b_{k+2}\frac{B_{j-k-1}}{2^{j-k-1}(j-k-1)!}\right)t^j.
\end{equation}
Equating coefficients of equal powers of $t$ in
\eqref{generalizedBernoullinumbers} and
\eqref{Chen-remainder-rmwriteas-t-2tthatisbj}, we see that
\begin{align}\label{Chen-remainder-rmwriteas-t-2tthatisbjseeThm3}
\frac{B_j}{j!}=\frac{B_j}{2^{j}\cdot
j!}+\sum_{k=0}^{j-1}b_{k+2}\frac{B_{j-k-1}}{2^{j-k-1}\cdot(j-k-1)!}\qquad
(j\in\mathbb{N}_0).
\end{align}
Substitution of the coefficients $b_j$ in \eqref{Write-Ej-as-bj} into
\eqref{Chen-remainder-rmwriteas-t-2tthatisbjseeThm3} then yields
\begin{align}\label{Chen-remainder-rmwriteas-t-2tthatisbjresultThm3}
B_j=\frac{j!}{2^{j}-1}\sum_{k=0}^{j-1}\frac{2(1-2^{k+2})B_{k+2}B_{j-k-1}}{(k+2)!\cdot
(j-k-1)!}\qquad (j\in \mathbb{N}).
\end{align}
It is easy to see that
\begin{align*}
B_j&=\frac{j!}{2^{j}-1}\left(\sum_{k=0}^{j-3}\frac{2(1-2^{k+2})B_{k+2}B_{j-k-1}}{(k+2)!\cdot
(j-k-1)!}-\frac{(1-2^j)B_j}{j!}+\frac{2(1-2^{j+1})B_{j+1}}{(j+1)!}\right)\\
&=\frac{j!}{2^{j}-1}\sum_{k=0}^{j-3}\frac{2(1-2^{k+2})B_{k+2}B_{j-k-1}}{(k+2)!\cdot
(j-k-1)!}+B_j+\frac{2(1-2^{j+1})B_{j+1}}{(2^{j}-1)(j+1)}.
\end{align*}
We therefore obtain
\begin{align*}
B_{j+1}=\frac{1}{2^{j+1}-1}\sum_{k=0}^{j-3}(1-2^{k+2})\binom{j+1}{k+2}B_{k+2}B_{j-k-1}           \quad \quad (n\in \mathbb{N}\setminus \{1, 2\}),
\end{align*}
which, upon replacing $j$ by $n-1$ and $k$ by $k-2$, yields \eqref{Thm3-Result}. This completes the proof of
Theorem \ref{Thm4-remainder}.
\end{proof}

\vskip 4mm
\begin{center}
{ Appendix:  An alternative proof of \eqref{remainder3}}
\end{center}
\setcounter{section}{1}
\setcounter{equation}{0}
\renewcommand{\theequation}{\Alph{section}.\arabic{equation}}
Euler's summation formula states that, for $k\in\mathbb{N}$, (see \cite[p.~9]{Temme1996})
\begin{align}\label{Euler-summation-formula}
f(1)&=\int_{0}^{1}f(x)\,{\rm d}
x+\sum_{j=1}^{k}\frac{(-1)^{j}B_{j}}{j!}\Big(f^{(j-1)}(1)-f^{(j-1)}(0)\Big)\nonumber\\
&\quad+\frac{(-1)^{k+1}}{k!}\int_{0}^{1}f^{(k)}(x)B_{k}(x)\,{\rm d}x.
\end{align}
 The choice $k=2n$ in \eqref{Euler-summation-formula} yields
\[\frac{f(1)+f(0)}{2}=\int_{0}^{1}f(x)\,{\rm d}
x+\sum_{j=1}^{n}\frac{B_{2j}}{(2j)!}\Big(f^{(2j-1)}(1)-f^{(2j-1)}(0)\Big)\hspace{3cm}\]
\[\hspace{5cm}-\frac{1}{(2n)!}\int_{0}^{1}f^{(2n)}(x)B_{2n}(x)\,{\rm d}x.\]
Application of this formula to $f(x)=e^{xt}$ then yields
\[\frac{t}{e^t-1}=1-\frac{t}{2}+\sum_{j=1}^{n}\frac{B_{2j}}{(2j)!}t^{2j}-\frac{t}{e^t-1}\frac{t^{2n}}{(2n)!}\int_{0}^{1}e^{xt}B_{2n}(x)\,{\rm d}x.\]

Using the following formula (see \cite[p.
 5]{Temme1996})
 \begin{equation*}
B_{2n}(x)= 2(-1)^{n+1}(2n)!\sum_{k=1}^{\infty}\frac{\cos(2k\pi
x)}{(2k\pi)^{2n}}\qquad (n\in\mathbb{N},\quad  0\leq x\leq1),
\end{equation*}
we have
\begin{align}\label{E2m-known-wehave}
\frac{t}{e^t-1}-\left(1-\frac{t}{2}+\sum_{j=1}^{n}\frac{B_{2j}}{(2j)!}t^{2j}\right)&=-\frac{t}{e^t-1}\frac{t^{2n}}{(2n)!}\int_{0}^{1}e^{xt}B_{2n}(x)\,{\rm d}x\nonumber\\
&=\frac{(-1)^{n}2t^{2n+1}}{e^t-1}\sum_{k=1}^{\infty}\int_{0}^{1}e^{xt}\frac{\cos (2k\pi
x)}{(2k\pi)^{2n}}\,{\rm d}x\nonumber\\
&=\frac{(-1)^{n}2t^{2n+2}}{(2\pi)^{2n}}\sum_{k=1}^{\infty}\frac{1}{k^{2n}(t^2+4\pi^2k^2)}\nonumber\\
&=(-1)^{n}t^{2n+2}\nu_{n}(t).
\end{align}
This gives another derivation of \eqref{remainder3}.

We obtain from \eqref{E2m-known-wehave} that
\begin{equation}\label{remainder-integral-Chen}
\nu_{n}(t)=\frac{(-1)^{n-1}}{t(e^t-1)\cdot(2n)!}\int_{0}^{1}e^{xt}B_{2n}(x)\,{\rm d}x,
\end{equation}
which provides an alternative integral representation of $\nu_{n}(t)$.

 \vskip 8mm


\enddocument